\documentclass[10pt]{article}
\usepackage{amsmath,amsthm,amsfonts,amssymb,amscd, amsxtra}

\newcommand{\banacha}{X}
\newcommand{\banachb}{Y}

\newtheorem{theorem}{Theorem}
\newtheorem{lemma}[theorem]{Lemma}

\newtheorem{corollary}[theorem]{Corollary}
\newtheorem{proposition}[theorem]{Proposition}

\newcommand{\norm}[1]{\left\|#1\right\|}
\newcommand{\normq}[1]{\norm{#1}^2}
\newcommand{\modulo}[1]{\left|#1\right|}

\begin{document}
\title{Kantorovich's  Theorem on Newton's Method}

\author{ O. P. Ferreira\thanks{IME/UFG, Campus II- Caixa
    Postal 131, CEP 74001-970 - Goi\^ania, GO, Brazil (E-mail:{\tt
      orizon@mat.ufg.br}). The author was supported in part by
    FUNAPE/UFG, PADCT-CNPq, PRONEX--Optimization(FAPERJ/CNPq), CNPq Grant 302618/2005-8,
    CNPq Grant 475647/2006-8 and
    IMPA.}  
  \and B. F. Svaiter\thanks{IMPA, Estrada Dona Castorina,
    110, Jardim Bot\^anico, CEP 22460-320 - Rio de Janeiro, RJ, Brazil
    (E-mail:{\tt benar@impa.br}). The author was supported in part by
    CNPq Grant 301200/93-9(RN), CNPq Grant 475647/2006-8 and by
    PRONEX--Optimization(FAPERJ/CNPq).}
}

\date{March 09, 2007}

\maketitle
\begin{abstract}
  In this work we present a simplifyed proof of Kantorovich's Theorem
  on Newton's Method.  This analysis uses a  technique which has
  already been used for obtaining new extensions of this theorem.

\noindent
\textsc{AMSC: 49M15, 90C30.}
\end{abstract}
\section{Introduction} \label{sec:int}
Kantorovich's Theorem assumes semi-local conditions to ensure
existence and uniqueness of a solution of a nonlinear equation
$F(x)=0$, where $F$ is a differentiable application between Banach
spaces~\cite{k48,bl8,O-AMM,Tapia}.  This theorem uses
\emph{constructively} Newton method and also guarantee convergence to
a solution of this iterative procedure.  Apart from the elegance of
this theorem, it has many theoretical and practical applications,
in~\cite{P04-2} we can find a reviews of recent applications and
in~\cite{potra2005} an application in interior point methods.  This
theorem has also many extensions, some of then encompassing previously
unrelated results see~\cite{ABM2004,W99}.  Some of these
generalizations and extensions are quite recent, because in the last
few year the Kantorovich's Theorem has been the subject of intense
research, see~\cite{ABM2004,FS02,FS06,P04-2,potra2005,W99}.

The aim of this paper is to present a new technique for the analysis of the
Kantorovich's Theorem.  This technique, was introduced in \cite{FS02} and
since then it has been used for obtaining new extensions of
Kantorovich's Theorem see~\cite{ABM2004,FS06}.  Here, it will be used to
present a simplified proof of its ``classical'' formulation.

The main idea  is to define ``good'' regions for Newton
method, by comparing the nonlinear function $F$ with its scalar
majorant function.  Once these good regions are obtained, an invariant
set for Newton method is also obtained and there,   Newton iteration
can be repeated indefinitely.

The following notation is used throughout our presentation.  Let
$\banacha$ be a Banach space. The open and closed ball at $x\in X$ are
denoted, respectively by
$$
B(x,r) = \{ y\in X ;\; \|x-y\|<r \}\;\;\; \mbox{and}\;\;\;
B[x,r] = \{ y\in X ;\; \|x-y\|\leqslant r\}.
\,
$$
For the Frechet derivative of a mapping $F$ we use the notation $F'$ and for the Dual space of $X$ we use $X^*$.

First, let us recall Kantorovich's  theorem on Newton's
method in its classical formulation, see~\cite{GT74,bl8,OR70,O96,potra2005}.

\begin{theorem}\label{th:knclass}
  Let $\banacha$, $\banachb$ be Banach spaces, $C\subseteq \banacha$
  and $F:{C}\to \banachb$ a continuous function, continuously
  differentiable on $\mathrm{int}(C)$. Take $x_0\in \mathrm{int}(C)$,
  $L,\, b>0$ and suppose that
  \begin{description}
  \item [1)]\hspace{2em} $F '(x_0)$ is non-singular,
  \item  [2)]\hspace{2em}
    \( \|
    F'(x_0)^{-1}\left[ F'(y)-F'(x)\right]
    \| \leq L\|x-y\|
    \)\;\;  for any $x,y\in C$,
  \item [3)] \hspace{2em}$ \|F'(x_0)^{-1}F(x_0)\|\leq b$,
  \item [4)]\hspace{2em} $2bL\leq 1$.
  \end{description}
  Define
  \begin{equation}
    \label{eq:scroots}
    t_*:=\frac{1-\sqrt{1-2bL}}{L},\qquad
    t_{**}:=\frac{1+\sqrt{1-2bL}}{L}. 
  \end{equation}
  If
  \[
  B[x_0,t_*]\subset C,
  \]
  then  the sequences $\{x_k\}$ generated by Newton's Method for
  solving $F(x)=0$ with starting point $x_0$,
  \begin{equation} \label{ns.KT} 
    x_{k+1} ={x_k}-F'(x_k) ^{-1}F(x_k), \qquad k=0,1,\cdots, 
  \end{equation}
  is well defined, is contained in $B(x_0, t_*)$, converges to a
  point $x_*\in B[x_0,t_*]$ which is the unique zero of $F$ in
  $B[x_0,t_*]$ and 
  \begin{equation}
    \label{eq:q.conv.x}
    \|x_*-x_{k+1}\|\leq \frac{1}{2} \|x_*-x_k \|, \qquad
    k=0,1,\,\cdots. 
  \end{equation}
   Moreover, if assumption {\bf 4} holds as an strict inequality, i.e.\
   $2bL<1$, then
   \begin{equation} \label{eq:q2}
     \|x_*-x_{k+1}\|\leq\frac{1-\theta^{2^k}}{1+\theta^{2^k}}
     \frac{ L}{2\sqrt{1-2bL}}\|x_*-x_k\|^2\leq
     \frac{ L}{2\sqrt{1-2bL}}\|x_*-x_k\|^2, \quad k=0,1,\cdots, 
   \end{equation}
  where $\theta:=t_*/t_{**}<1$, and $x_*$ is the
  unique zero of $F$ in $B[x_0,\rho]$ for any $\rho$ such that
  \[ t_*\leq\rho<t_{**},\qquad B[x_0,\rho]\subset C.\]
\end{theorem}
\noindent
Note that under assumption {\bf 1}-{\bf 4}, convergence of $\{x_k\}$
is $Q$-linear, according to \eqref{eq:q.conv.x}. The additional 
assumption $2bL<1$ guarantee $Q$-quadratic convergence, according to
\eqref{eq:q2}.  This additional assumption also guarantee that $x_*$ is
the unique zero of $F$ in $B(x_0,t_{**})$, whenever $B(x_0,t_{**})\subset C$.

>From now on, we assume that the hypotheses of Theorem \ref{th:knclass}
hold, with the exception of $2bL< 1$ which will be considered to
hold only when explicitly stated.

\section{ Kantorovich's Theorem for a scalar quadratic function}
In this section we  analyze Newton method applied to solve the 
scalar equation $f(t)=0$, for   $f:\mathbb{R}\to \mathbb{R}$  
\begin{equation}
  \label{eq:def.f}
  f(t)= \frac{L}{2}t^2-t+b.
\end{equation}
The analysis to be performed can also be viewed as Kantorovich's
theorem for function $f$.  This function and the sequence generated by
Newton method for solving $f(t)=0$ with starting point $t_0$,
\begin{equation}
  \label{eq:def.t0}
   t_0:=0,
\end{equation}
both will play an important rule in the analysis of
Theorem~\ref{th:knclass}.

Note that the assumptions of Theorem \ref{th:knclass} are satisfied
in the \emph{very} particular case
 $F=f$, $X=Y=C=\mathbb{R}$, $x_0=t_0$.  The roots of $f$ are $t_{*}$ and
$t_{**}$, as defined in \eqref{eq:scroots}.  As $b,\, L>0$,
\[
 0<t_{*}\leq t_{**},
\]
with strict inequality between $t_*$ and $t_{**}$ if and only if $2bL<1$. 
Hence
\begin{itemize}
\item $t_*$ is the unique root of $f$ in $B[t_0, t_*]$,
\item if $2bL<1$, then $t_*$ is the unique root of $f$ in $B(t_0, t_{**})$.
\end{itemize}
So, the existence  and uniqueness part of Theorem
\ref{th:knclass} for zeros of $f$ holds.
\begin{proposition}
  \label{pr:scalar.1}
  The scalar function $f$ has a smallest nonnegative root 
  $t_*\in (0,1/L]$.  Moreover, for any $t\in [0,t_*)$
  \[
  f(t)>0, \qquad f'(t)\leq  L(t-t_*)<0 .
  \]
\end{proposition}
\begin{proof}
  For the first statement, it remains to prove that $t_*\leq 1/L$,
  which is a trivial consequence of the assumptions on $b$ and $L$.

  As $f(0)>0$, $f$ shall be strictly positive in $[0,t_*)$.  For the
  last inequalities, use the inequality $t_*\leq 1/L$ and  \eqref{eq:def.f}
  to obtain
  \[
  f'(t)=Lt-1=L(t-1/L)\leq L(t-t_*).
  \]
  Now, the last inequality follows directly  from the assumption  $t<t_*$.
\end{proof}
According to Proposition \ref{pr:scalar.1}, $f'(t)\neq 0$ for all
$t\in[0,t_*)$.  Therefore, Newton iteration is well defined in
$[0,t_*)$. Let us call it $n_f$,
\begin{equation} \label{eq:def.nfs}
  \begin{array}{rcl}
  n_f:[0,t_*)&\to& \mathbb{R}\\
    t&\mapsto& t-f(t)/f'(t).
  \end{array}
\end{equation}
Note that, up to now,  only \emph{one single iteration} of newton
method is well defined in $[0,t_*)$.  In principle,  Newton iteration  
could map some $t\in[0,t_*)$ in to $1/L$.  In such a case, the second iterate
for  $t$ would be not defined.

Now, we shall prove that Newton iteration can be repeated indefinitely
at any starting point in $[0,t_*)$.
\begin{proposition} \label{pr:2}
  For any $t\in [0,t_*)$
  \[
  t_*-n_f(t)=-\frac{L}{2f'(t)}(t_*-t)^2, \qquad
  t<n_f(t)<t_*.
  \]
  In particular, $n_f$ maps  $[0,t^*)$ in
  $[0,t^*)$.
\end{proposition}
\begin{proof}
  Take $t\in [0,t_*)$.  As $f$ is a second-degree polynomial and $f(t_*)=0$,
  \[ 
  0=f(t)+f'(t)(t_*-t) +\frac{L}{2}(t_*-t)^2.
  \]
  Dividing by $f'(t)$ we obtain, after direct rearranging 
  \[
    t_*-t+f(t)/f'(t)=-\frac{L}{2f'(t)}(t_*-t)^2 .
  \]
  Note that, by \eqref{eq:def.nfs}, the left hand side of the above
  equation is $t_*-n_f(t)$, which proves the first equality.
  
  Using Proposition \ref{pr:scalar.1} we have $f(t)>0$ and $f'(t)<0$.  
  Combining these inequalities with definition
  \eqref{eq:def.nfs} and the first equality in the proposition,
  respectively, we obtain $t<n_f(t)<t_*$.
  The last statement of the Proposition follows directly from these
   inequalities.
\end{proof}
Proposition~\ref{pr:2} shows, in particular, that
for any $t\in [0,t_*)$, the sequence $\{n_f ^k(t)\}$,
\[
 n_f^{0}(t)=t,\quad n_f^{k+1}(t)= n_f\left(n_f^k \left(t\right)\right),
  \qquad k=0,1,\cdots.
\]
is well defined, strictly increasing, remains in $[0,t_*)$ and so, is
convergent.  Therefore, Newton method for solving $f(t)=0$ with starting
point $t_0=0$ ( see \eqref{eq:def.t0}) generates an infinite sequence
$\{t_k=n_f^k(t_0)\}$, which can be also defined as
\begin{equation}\label{eq:tknk}
  t_0=0,\quad t_{k+1}=n_f(t_k), \qquad k=0,1,\cdots.
\end{equation}
As we already observed, this sequence is strictly increasing, remains
in $[0, t_*)$ and converges.
\begin{corollary} \label{cr:kanttk} The sequence $\{t_k\}$ is well
  defined, strictly increasing and is contained in $[0,t_*)$.
  Moreover, it converges $Q$-linearly to $t_*$, as follows
  \begin{equation}\label{eq:q.conv.f}
    t_*-t_{k+1}=\frac{L}{-2f'(t_k)}(t_*-t_k)^2\leq \frac{1}{2} (t_*-t_k), \qquad
    k=0,1,\,\cdots.
  \end{equation}
  If $2bL< 1$, then the sequence $\{t_k\}$ converge
  $Q$-quadratically as follows
  \begin{equation} \label{eq:q2f}
    t_*-t_{k+1}=\frac{1-\theta^{2^k}}{1+\theta^{2^k}}
    \frac{ L}{2\sqrt{1-2bL}}(t_*-t_k)^2\leq
    \frac{ L}{2\sqrt{1-2bL}}(t_*-t_k)^2,
    \qquad k=0,1,\cdots,
  \end{equation}
  where $\theta=t_*/t_{**}<1$.
\end{corollary}
\begin{proof}
  The first statement of the corollary have already been proved.
 
  Using  Proposition \ref{pr:2}, we have for any $k$
  \[
    t_*-n_f(t_{k})=\frac{L}{-2f'(t_k)}(t_*-t_k)^2,
  \]
  which combined with \eqref{eq:tknk}, yields the equality on
  \eqref{eq:q.conv.f}.  As $t_k\in [0, t_*)$, using Proposition
  \ref{pr:scalar.1} we have
  \[
  f'(t_k)\leq L(t_k-t_*)<0.
  \]
  The multiplication of the first above inequality by $(t_*-t_k)/(2f'(t_k))<0$
  yields the inequality in \eqref{eq:q.conv.f}.

  Now suppose that $2bL< 1$ or equivalently $t_*< t_{**}$.  A closed
  expression for $t_k$ is available ( see, e.g. [\cite{O96}, Appendix
  F], \cite{GT74}) see the Appendix A. In this case
  \[
  t_k=t_*- \frac{\theta^{2^k}}{1-\theta^{2^k}}\frac{2\sqrt{1-2bL}}{ L}, \qquad
  k=0,1,\,\cdots.
  \]
  From above equation we have that
  \[
  f'(t_k)=-\frac{1+\theta^{2^k}}{1-\theta^{2^k}}\frac{\sqrt{1-2bL}}{a}, \qquad
  k=0,1,\,\cdots.
  \]
  Therefore, to obtain the equality in \eqref{eq:q2f} combine the
  equality in \eqref{eq:q.conv.f} and latter equality.  As
  $(1-\theta^{2^k})/(1+\theta^{2^k})\leq 1$ the inequality in
  \eqref{eq:q2f} follows.
\end{proof}

\section{Simplifying assumption and convergence}

Newton method is invariant under (non-singular) linear
transformations.  This fact will be used to simplify our analysis.
We claim that it is enough to prove Theorem~\ref{th:knclass} for the case
$X=Y$ and $F'(x_0)=I$.  Indeed, if $F'(x_0)\neq I$, define
\[ G=F'(x_0)^{-1}F.\] Then, the domain, the roots, the domain of the
derivative and the points where the derivative is non-singular are the same for $F$ and $G$.  Moreover, Newton method
applied to $F(x)=0$ is equivalent to
Newton methods applied to $G(x)=0$, i.e., at the points where $F'(x)$
is nonsingular,
\[
 G'(x)^{-1}G(x)=F'(x)^{-1}F(x),
 \qquad x-G'(x)^{-1}G(x)=x-F'(x)^{-1}F(x).
\]
  Finally,
 $G$ will satisfy the \emph{same} assumptions wich $F$ satisfy.
So, from now one we assume
\begin{equation}
  \label{eq:simplify}
  X=Y,\qquad  F'(x_0)=I.
\end{equation}
Note that this assumption  simplifies conditions {\bf 2} and 
{\bf 3} of Theorem~\ref{th:knclass}.
\begin{proposition}
  \label{k1}
  If  $0\leq t<t_*$  and $x\in B(x_0, t)$, then
  $F'(x)$ is non-singular and
  \[
  \norm{  F'(x) ^{-1} }\leq
  1/\modulo{f'(t)}.
  \]
\end{proposition}
\begin{proof}
  Recall that $t_*\leq 1/L$.  Hence, $0\leq t<1/L$. Using
  \eqref{eq:simplify} and assumption {\bf 2}, with $x$ and $x_0$ we
  have
  \[
  \norm{ F'(x)-I}\leq L t<1.
  \]
  Hence, using Banach's Lemma, we
  conclude that $ F'(x)$ is non-singular and
  \[   \norm{  F'(x)^{-1} }\leq 1/(1-Lt).\]
  To end the proof, use \eqref{eq:def.f}  to obtain
  $\modulo{f'(t)}=1-Lt$ for $0\leq t<t_*$.
\end{proof}
The error in the first order approximation of $F$ at point $x\in \mathrm{int}(C)$ can be
estimated in any $y\in C$, whenever the line segment with extreme
points $x,y$ lays in $C$.  Since balls are convex, we have:
\begin{proposition}
  \label{k2}
  If  $x\in B(x_0,R)$ and $y\in B[x_0,R]\subset C$, then
  \[   \norm{ F(y)-\big[F(x)+F'(x)(y-x)\big] }
       \leq \frac{L}{2}\|y-x\|^2.\]
\end{proposition}
\begin{proof}
  Define, for $\theta\in [0,1]$,
  \[ y(\theta)=x+\theta(y-x),\qquad 
     R(\theta)= 
         F(y(\theta))-\left[F(x)+F'(x)(y(\theta)-x) \right].
       \]
  We shall estimate $\norm{R(1)}$. From Hahn-Banach Theorem, there exists $\xi\in X^*$ such that
  \[ \norm{\xi}=1,\qquad \xi(R(1))=\norm{R(1)}.\]
  Define, for $\theta\in[0,1]$,
  \[ g(\theta)=\xi(R(\theta)).\]
 Direct calculation yields, for $\theta\in [0,1)$
  \[ \frac{dg}{d\theta}\,(\theta)
    =\xi\bigg(F'(y(\theta))
    -F'(x)\bigg).\]
  In particular, $g$ is $C^1$ on $[0,1)$.
  Using assumption {\bf 2}, we have
  \[ \frac{dg}{d\theta}\,
  (\theta)\leq L\theta\norm{y-x}.\]
  To end the prove, note that $\xi(R(0))=0$ and perform direct
  integration on the above inequality.
\end{proof}
Proposition \ref{k1}  guarantee 
non-singularity of $F'$, and so 
well definedness of Newton iteration map
for solving $F(x)=0$ in $B(x_0,t_*)$.  Let us call $N_{F}$ the Newton
iteration map (for $F(x)=0$) in that region
\begin{equation} \label{eq:def.NF}
  \begin{array}{rcl}
  N_F:B(x_0,t_*)&\to& X\\
    x&\mapsto& x-F'(x)^{-1}F(x).
  \end{array}
\end{equation}
One can apply a \emph{single} Newton iteration on any $x\in
B(x_0,t_*)$ to obtain $N_{F}(x)$
which may not belong
to $B(x_0,t_*)$,
 or even may not belong to the domain of $F$. 
To ensure that Newton
iterations  may be repeated indefinitely from $x_0$, we need some
additional results.

First, define some subsets of $B(x_0, t_*)$ in which, as we shall
prove, Newton iteration \eqref{eq:def.NF} is ``well behaved''.
\begin{align}\label{E:K}
K(t)&:=\left\{ x\in B[x_0,t] \, : \;  \norm{F(x)}\leq f(t)
   \right\},\qquad
   t\in [0,t_*)\,,\\
  \label{eq:def.K}
 K&:=\bigcup_{t\in[0,t_*)} K(t).
\end{align}
\begin{lemma}
  \label{lm:k}
  For any $t\in [0,t_*)$ and $x\in K(t)$,
  \begin{enumerate}
  \item $ \norm{F'(x)^{-1}F(x)}\leq -f(t)/f'(t)$,
  \item $\norm{x_0-N_F(x)}\leq n_f(t)$,
  \item $\norm{F(N_F(x))}\leq f(n_f(t))$.
  \end{enumerate}
  In particular,
  \[   N_F(K(t))\subset K(n_f(t)),\qquad \forall\; t\in [0,t_*)\]
  and $N_F$ maps $K$ in $K$, i.e., $ N_F(K)\subset K$.
\end{lemma}
\begin{proof}
  Take $x\in K(t)$. Using Proposition \ref{k1} and \eqref{E:K}  we conclude that 
   $F'(x)$ is non-singular,
  \[  
   \norm{ F'(x)^{-1}}\leq 1/\modulo{f'(t)},
    \qquad
   \norm{F(x)}\leq f(t).
   \]
   Hence,
   \[ \norm{ F'(x)^{-1} F(x)}\leq
   \norm{ F'(x)^{-1}}\,  \norm{F(x)}
   \leq f(t)/\modulo{f'(t)},
   \]
   which combined with the inequality $f'(t)<0$ yields item 1.
   
   To prove item 2 use item 1,  triangular inequality and definition
   \eqref{eq:def.NF} to obtain 
   \[
   \norm{x_0-N_F(x)}\leq \norm{x_0-x}+\norm {F'(x)^{-1}F(x)}
   \leq t-f(t)/f'(t).
   \]
   To end the prove of item 2, combine the above equation with
   definition \eqref{eq:def.nfs}.

    From item 2 and Proposition \ref{pr:2},  $N_F(x)\in B(x_0,t_*)$.
    So,  Proposition~\ref{k2} implies
    \begin{eqnarray*}
       \norm{F(N_F(x))-\big[F(x)+F'(x)\left(N_F(x)-x\right)\big]}\leq
       \frac{L}{2}\normq{F'(x)^{-1}F(x)}.
    \end{eqnarray*}
    Note that by \eqref{eq:def.NF} 
    \[ F(x)+F'(x)\left(N_F(x)-x\right)=0.\]
    Combining last two equations, item 1 and identity $f(n_f(t))
    =L(f(t)/f'(t))^2/2$ (which follows from \eqref{eq:def.f} and \eqref{eq:def.nfs})
    we conclude that item 3 also holds. 

    Since $t<n_f(t)<t_*$ (Proposition \ref{pr:2}), using also items 2
    and 3 we have that 
    \[
    N_F(x)\in K(n_f(t)).
    \]
    As $x$ is an
    arbitrary element of $K(t)$, we have $N_F(K(t))\subset K(n_f(t))$.

   To prove the last inclusion, take $x\in
   K$.  Then $x\in K(t)$ for some $t\in[0,t_*)$, which readily implies
   $N_F(x)\in K(n_f(t))\subset K$. 
\end{proof}
The last inclusion in Lemma \ref{lm:k} shows that for any $x\in K$,
the sequence  $\{N_F^k(x)\}$,
\[ N_F^{0}(x)=x,\quad N_F^{k+1}(x)= N_F\left(N_F^k \left(x\right)\right),
  \qquad k=0,1,\cdots,
\]
is well defined and remains in $K$.  The assumptions of Theorem
\ref{th:knclass} guarantee
\begin{equation}
  \label{eq:x.in.k0}
  x_0\in K(0)\subset K.
\end{equation}
Therefore, the sequence $\{x_k=N_F^k(x_0)\}$ is well defined and
remains in $K$.  This sequence can be also defined as
\begin{equation}\label{eq:xknk}
  x_0=0,\quad x_{k+1}=N_F(x_k), \qquad k=0,1,\cdots.
\end{equation}
which happens to be the same sequence specified in \eqref{ns.KT},
Theorem~\ref{th:knclass}.
\begin{proposition}
  \label{pr:conv.xk}
  The sequence $\{x_k\}$ is well defined, is contained in
  $B(x_0,t_*)$  and
  \begin{equation}
    \label{eq:inducao}
     x_k\in K(t_k),\qquad k=0,1,\cdots.
  \end{equation}
  Moreover,  $\{x_k\}$ converges to a point $x_*\in B[x_0,t_*]$,
  \[
   \norm{x_*-x_k}\leq t_*-t_k,\qquad k=0,1,\cdots,
  \]
  and   $F(x_*)=0$.
\end{proposition}
\begin{proof}
  Well definedness of the sequence $\{x_k\}$ was already proved.  We
  also conclude that this sequence remains in $K$.  As $K\subset
  B(x_0,t_*)$ (see \eqref{E:K} and \eqref{eq:def.K}), $\{x_k\}$ also
  remains in $B(x_0,t_*)$.

  As $t_0=0$, the first inclusion in \eqref{eq:x.in.k0} can also be
  written as $x_0\in K(t_0)$.  So, \eqref{eq:inducao} holds for $k=0$.
  To complete the proof of \eqref{eq:inducao} use induction in $k$, 
  \eqref{eq:xknk}, Proposition~\ref{lm:k} and equation \eqref{eq:tknk}.

  Combining \eqref{eq:inducao} with item 1 of Lemma \ref{lm:k},
  \eqref{eq:xknk} and \eqref{eq:tknk} we obtain
  \begin{equation}
    \label{eq:norm.dx.dt}
    \norm{x_{k+1}-x_k}\leq t_{k+1}-t_k,\qquad k=0,1,\cdots.
  \end{equation}
  As $\{t_k\}$ converges and $\sum_{k=0} ^\infty t_{k+1}-t_k<\infty$ we conclude that
  $\{x_k\}$ is a Cauchy sequence.  So, $\{x_k\}$ converges to some
  $x_*\in B[x_0,t_*]$.  Moreover, 
  \eqref{eq:norm.dx.dt} implies
  \begin{equation}
    \label{eq:x0xkx*}
    \norm{x_*-x_k}\leq \sum_{j=k} ^\infty t_{j+1}-t_j=t_*-t_k,
    \qquad 
    \qquad k=0,1,\cdots.
  \end{equation}
  Note that
  \[  F(x_k)=F'(x_{k-1})[x_{k-1}-x_k].\]
  As $\norm{F'(x)}$ is bounded by $1+Lt_*$ in $B(x_0,t_*)$ last equation implies
  that
  \[ \lim_{k\to\infty} F(x_k)=0.\]
  Now, using the continuity of $F$ in $B[x_0,t_*]$ we have that
  $F(x_*)=0$.
\end{proof}

\section{Uniqueness and  convergence rate}
To prove uniqueness and estimate the convergence rate, another
auxiliary result will be needed.
\begin{proposition}
  \label{pr:rate.uniq}
  Take $x,y \in X$, $t,v\geq 0$. If
  \[
  \norm{x-x_0}\leq t<t_*,\quad  \norm{y-x_0}\leq R,\quad F(y)=0,
  \quad f(v)\leq 0
  \]
  and $B[x_0,R]\subset C$, then
  \[
  \norm{y-N_F(x)}
        \leq  
         \left[v-n_f(t)\right]\frac{\normq{y-x}}{(v-t)^2}.
  \]
\end{proposition}
\begin{proof}
Note from \eqref{eq:def.NF} 
\[
y-N_F(x)=F'(x)^{-1}[F(x)+F'(x)(y-x)].
\]
As $F(y)=0$,  using also Proposition \ref{k2} we obtain
  \[
 \norm{F'(x)^{-1}[F(x)+F'(x)(y-x)]}\leq 
  \frac{L}{2}\|y-x\|^2,
  \]
  and from Proposition  \ref {k1}
  \[ 
  \norm{F'(x)^{-1}}\leq 1/\modulo{f'(t)}.
  \]
  Combining these equations we have
  \[
    \norm{y-N_F(x)}\leq \frac{L}{2\modulo{f'(t)}} (v-t)^2\;
  \frac{\normq{y-x}}{ (v-t)^2}.
 \]
 As $f'(t)<0$ and $f(v)\leq 0$, using also \eqref{eq:def.nfs} we have
 \begin{align*}
   v-n_f(t)
   &=\frac{1}{-f'(t)} [-f(t)-f'(t)(v-t)]\\
    &\geq \frac{1}{\modulo{f'(t)}} [f(v)-f(t)-f'(t)(v-t)]=
 \frac{L}{2\modulo{f'(t)}}(v-t)^2.
 \end{align*}
 Combining the two above inequalities we obtain the desired result.
\end{proof}

\begin{corollary}
  \label{cr:rt.cv}
  If $y\in B[x_0,t_*]$ and $F(y)=0$, then
  \[     \|y-x_{k+1}\|\leq \frac{t_*-t_{k+1}}{(t_*-t_k)^2}\; 
  \|y-x_k\|^2, 
  \qquad \norm{y-x_k}\leq t_*-t_k,\qquad k=0,1,\cdots.
  \]
 In particular, $x_*$ is the unique zero of $F$ in $B[x_0,t_*]$.
\end{corollary}
\begin{proof}
  Take an arbitrary $k$. From Proposition~\ref{pr:conv.xk} we have $x_k\in K(t_k)$. So,  $\norm{x_k-x_0}\leq t_k$ and we can apply Proposition~\ref{pr:rate.uniq} with
  $x=x_k$,  $t=t_k$ and $v=t_*$, to obtain
  \[ \|y-N_F(x_{k})\|\leq [t_*-n_f(t_{k})] \;
  \frac{\|y-x_k\|^2}{(t_*-t_k)^2}.
  \]
  The first inequality now follows from the above inequality,
  \eqref{eq:xknk} and \eqref{eq:tknk}.

  We will prove the second inequality by induction.  For $k=0$ this
  inequality holds, because $y\in B[x_0,t_*]$ and $t_0=0$.  Now,
  assume that the inequality holds for some $k$,
  \[
  \norm{y-x_k}\leq t_*-t_k.
  \]
  Combining the above inequality with the first inequality of the
  corollary, we have that $\norm{y-x_{k+1}}\leq t_*-t_{k+1}$, wich
  concludes the induction.

  We already know that $x_*\in B[x_0,t_*]$ and $F(x_*)=0$.  Since $\{x_k\}$
  converges to $x_*$ and $\{t_k\}$ converges to $t_*$, using the
  second inequality of the corollary we conclude $y=x_*$.  Therefore,
  $x_*$ is the unique zero of $F$ in $B[x_0,t_*]$.
\end{proof}
\begin{corollary}
  \label{cr:lc}
  The sequences $\{x_k\}$ and $\{t_k\}$  satisfy 
  \begin{equation}  \label{eq:cla}
    \|x_*-x_{k+1}\|\leq \frac{t_*-t_{k+1}}{(t_*-t_k)^2}\; 
    \|x_*-x_k\|^2, \qquad  k=0,1,\cdots.
  \end{equation}
  In particular,
  \begin{equation} \label{eq:lc1}
    \|x_*-x_{k+1}\|\leqslant \frac{1}{2} \|x_*-x_k \|, 
    \qquad  k=0,1,\cdots.
  \end{equation}
  Additionally, if   $2bL< 1$ then
  \begin{equation} \label{eq:qc1}
    \|x_*-x_{k+1}\| \leq\frac{1-\theta^{2^k}}{1+\theta^{2^k}}
    \frac{ L}{2\sqrt{1-2bL}}\|x_*-x_k\|^2\leq
    \frac{ L}{2\sqrt{1-2bL}}\|x_*-x_k\|^2, 
    \quad k=0,1,\cdots.
  \end{equation}
\end{corollary}
\begin{proof}
  According to Proposition~\ref{pr:conv.xk}, $x_*\in B[x_0,t_*]$ and
  $F(x_*)=0$.  To prove equation \eqref{eq:cla} apply
  Corollary~\ref{cr:rt.cv} with $y=x_*$.

  Note that, by \eqref{eq:q.conv.f} in Corollary~\ref{cr:kanttk},
   and Proposition~\ref{pr:conv.xk}, for any $k$
  \[
  (t_*-t_{k+1})/(t_*-t_k)\leq 1/2 \quad \mbox{ and } 
  \quad \|x_*-x_k\|/(t_*-t_k)\leq 1.
  \]
  Combining these inequalities with \eqref{eq:cla} we have
  \eqref{eq:lc1}.  Now, assume that $bL< 1/2$ holds. Then,
  \eqref{eq:q2f} in Corollary \ref{cr:kanttk} and \eqref{eq:cla} imply
  \eqref{eq:qc1} and the corollary is proved.
\end{proof}
\begin{corollary}
 \label{uniq}
 If  $2bL< 1$, $t_*\leq \rho <t_{**}$ and $B[x_0,\rho]\subseteq
 C$ then $x_*$ is the unique zero of $F$ in $B[x_0,\rho]$.
\end{corollary}
\begin{proof}
Assume that there exists $y_*\in C$ such that
$\norm{y_*-x_0}<\rho$ and $ F(y_*)=0.$  
Using Proposition~\ref{k2} with $x=x_0$ and $y=y_*$ (recall that
$F'(x_0)=I$ ) we obtain that
\[
\norm{F(x_0)+y-x_0} \leqslant\frac{L}{2}\norm{y-x_0}^2.
\]
Triangle inequality and assumption {\bf 3} of
Theorem~\ref{th:knclass} yield
\[
\norm{F(x_0)+y-x_0}\geq \norm{y-x_0}-\norm{F(x_0)}\geq \norm{y-x_0}-b.
\]
Combining the above inequalities we obtain
 \[
\frac{L}{2}\norm{y-x_0}^2\geq   \norm{y_*-x_0}-{b},
 \] 
which is equivalent to $f(\norm{y_*-x_0})\geq 0$. As $\norm{y_*-x_0} \leq\rho<t_{**}$ last inequality implies that $\norm{y_*-x_0}\leq t_*$.  Therefore, from Corollary \ref{cr:rt.cv} and assumption $F(y_*)=0$,  we conclude that $y_*=x_*$.
\end{proof}
Therefore, it follows from  Proposition \ref{pr:conv.xk},  
Corollary  \ref{cr:rt.cv},  Corollary \ref{cr:lc} and Corollary \ref{uniq} that all statements in Theorem \ref{th:knclass} are valid.
\subsection{Appendix: A closed formula for $t_k$}
Note that
$f(t)=(L/2)(t-t_*)(t-t_{**})$, and $f'(t)=(L/2)[(t-t*)+(t-t_{**})].$
Using the above equations and \eqref{eq:tknk}, 
\[
 t_{k+1}-t_*=(t_k-t_*)-\frac{(t_k-t_*)(t_k-t_{**})}{(t_k-t*)+(t_k-t_{**})}
 =\;\frac{(t_k-t_{*})^2}{(t_k-t*)+(t_k-t_{**})}.
\]
By similar manipulations, we have
\[
 t_{k+1}-t_{**}=\frac{(t_k-t_{**})^2}{(t_k-t*)+(t_k-t_{**})}.
\]
Combing two latter equality we obtain that
\[
\frac{t_{k+1}-t_{*}}{t_{k+1}-t_{**}}
=\left(\frac{t_{k}-t_{*}}{t_{k}-t_{**}}\right)^2.
\]
Suppose that $2bL<1$.  In this case, $t_*<t_{**}$. Hence, using the definition
$
 \theta:=t_*/t_{**}<1,
$
and induction  in $k$ we have
\[\frac{t_{k}-t_{*}}{t_{k}-t_{**}}=\theta^{2^{k}}.\]
After some algebraic manipulation in
above equality we obtain hat
\[
t_{k}=\frac{t_{**}\theta^{2^k}-t_{*}}{\theta^{2^k}-1}=t_*-
\frac{\theta^{2^k}}{1-\theta^{2^k}}\frac{2\sqrt{1-2bL}}{ L}, \qquad
k=0,1,\,\cdots.
\]

\end{document}